\documentclass{amsart}[10pt]
\usepackage{amsmath}
\usepackage{amssymb} 
\usepackage{epsbox}
\usepackage{amsthm}
\usepackage{graphicx}

 \newtheorem{thm}{Theorem}
 \newtheorem{cor}{Corollary}
 \newtheorem{prop}{Proposition}
 \newtheorem{lem}{Lemma}

 {\theoremstyle{definition}
 \newtheorem{rem}{Remark}}
  {\theoremstyle{definition}
 \newtheorem{defn}{Definition}}
  {\theoremstyle{definition}
 \newtheorem{exam}{Example}}
 {\theoremstyle{definition}
 \newtheorem{conj}{Conjecture}}

\begin{document}  
\newcommand{\R}{\mathbb{R}}
\newcommand{\Q}{\mathbb{Q}}
\newcommand{\Z}{\mathbb{Z}}
\newcommand{\ttwo}{\overline{t'_{2}}}
\newcommand{\mQ}{\mathcal{Q}}
\newcommand{\mW}{\mathcal{W}}
\newcommand{\CB}{\mathcal{CB}}
\newcommand{\wW}{\widetilde{W}}

\newcommand{\red}{\textrm{red}\,}

\title{Non-left-orderable double branched coverings}
\author{Tetsuya Ito}
\address{Graduate School of Mathematical Science, University of Tokyo, 3-8-1 Komaba Meguro-ku Tokyo 153-8914, Japan}
\email{tetitoh@ms.u-tokyo.ac.jp}
\urladdr{http://ms.u-tokyo.ac.jp/~tetitoh}
 
\begin{abstract}
We develop a method to show the fundamental group of the double branched covering of a link is not left-orderable by introducing the notion of the coarse presentation. As in the usual group presentations, a coarse presentation is given by a set of generators and relations, but inequalities are allowed as relations. By using coarse presentation, we give a family of links whose double branched covering has the non-left-orderable fundamental group. Our family of links includes many known examples and new examples.
\end{abstract}
\maketitle

\section{Introduction}

Let $G$ be a group. A {\em left ordering} of $G$ is a total ordering $<_{G}$ of $G$ which is preserved by the left action of $G$. That is, $ a<_{G} b$ implies $ga <_{G} gb$ for all $g,a,b\in G$.
$G$ is called {\em left-orderable} if $G$ admits at least one left ordering.
We adapt the convention that the trivial group $G=\{1\}$ is {\em not} left-orderable. 
As we will see in Section 5, this convention is natural when we study the relation between left orderings and topology of 3-manifolds.

For a link $L$ in $S^{3}$, let $\Sigma_{2}(L)$ be the double branched covering of $L$.
In this paper we study the non-existence of left orderings of $\pi_{1}(\Sigma_{2}(L))$, the fundamental group of the double branched covering. The aim of this paper is to give a new method to show $\pi_{1}(\Sigma_{2}(L))$ is {\em not} left-orderable, by introducing a new notion called a {\em coarse presentation}.  

A coarse presentation is a generalization of a group presentation. Like a usual group presentation, a coarse presentation is given as a set of generators and a set of relations. The main difference is that a coarse presentation allows to include {\em inequalities} as its relations.

We construct a coarse presentation associated to a link diagram $D$ which we will call the {\em coarse Brunner's presentation}. Based on the coarse Brunner's presentation, we will give various families of links whose double branched covering has non-left-orderable fundamental group.
In Theorem \ref{thm:simalter}, we will treat links represented by diagrams which is similar to alternating link diagrams in the coarse presentation point of view. This family of links contains all alternating links, hence extends the results obtained in \cite{bgw},\cite{gr}: the double branched covering of an alternating link has the non-left-orderable fundamental group.
We will also treat various non-alternating links as well. In Theorem \ref{thm:main} and Theorem  \ref{thm:main2}, we will show that the double branched covering of links represented by some particular diagrams have the non-left-orderable fundamental group. These family of links contain all positive knots of genus two and many non-alternating links. The reader should regard these Theorems rather as examples of coarse presentation arguments. In a similar manner, the coarse presentation allows us to find a lot of other examples of links whose branched double covering have non-left-orderable fundamental groups.

Although it is possible to prove all results in this paper by using the usual group presentations, the coarse presentation argument has several benefits. First of all, the coarse presentation is much simpler than the usual group presentation: it has less generators and less relations than the usual group presentation. In particular, by using coarse presentation, the proof of non-left-ordrability becomes much simpler compared with the proof based on the usual group presentation. 
 Moreover, in the proof of non-left-orderability, the coarse presentation allows us to separate the role of the link diagrams into the {\em local} properties and the {\em global} properties. 
That is, our arguments of non-left-orderability are valid if we replace a crossing with an (algebraic) tangle which has a property similar to the original crossing. Thus the coarse presentation argument provides more unified point of view in the proof of non-left-orderability.

Let us sketch here a main idea of the coarse presentation argument.
Let $D$ be a diagram representing a link $L$.
In \cite{br}, Brunner constructed a presentation of $\pi_{1}(\Sigma_{2}(L))$ from the diagram $D$. 
This presentation has a lot of generators and relations. We try to extract essential information which is sufficient to show $\pi_{1}(\Sigma_{2}(L))$ is not left-orderable. 

For an algebraic tangle $A$ in the diagram $D$, we associate special elements $W_{A},R,L$ of $\pi_{1}(\Sigma_{2}(L))$, which will serve as generators of the coarse presentation. By studying Brunner's presentation of $\pi_{1}(\Sigma_{2}(L))$ precisely, we find the commutative relation $(L^{-1}R) W_{A} = W_{A} (L^{-1}R)$ and an inequality $(L^{-1}R)^{m} \leq W_{A}^{N} \leq (L^{-1}R)^{M}$ which are valid for {\em all} left orderings of $\pi_{1}(\Sigma_{2}(L))$. We will express this commutative relation and ``universal" inequalities by $W_{A} \in [[\frac{m}{N}, \frac{M}{N}]]_{L^{-1}R}$. This ``universal" inequality is obtained from only the tangle diagram $A$. This is what we referred as the ``local properties".

The coarse presentation is a collection of a set of generators, a set of ``universal" inequalities, and a set of relations given as usual equalities, which corresponds to what we referred as the ``global properties". The usual equalities describe how algebraic tangles are connected in the whole link diagram. 
These two kinds of information allow us to show that $\pi_{1}(\Sigma_{2}(L))$ is not left-orderable: If $\pi_{1}(\Sigma_{2}(L))$ has left-ordering, we deduce a contradiction from the coarse presentation.

The plan of this paper is as follows. In Section \ref{sec:pre} we review standard notions of tangles, and a presentation of $\pi_{1}(\Sigma_{2}(L))$ due to Brunner \cite{br}, which is our starting point.
We construct the coarse Brunner's presentation in Section 3. The construction of coarse Brunner's presentation consists of three steps: First we deduce an intermediate group presentation that simplifies the Brunner's presentation, which we call the {\em reduced Brunner's presentation}. 
In the second step we define the tangle element $W_{A}$ for an algebraic tangle in a link diagram, and study its fundamental properties. The third step consists of establishing the ``universal" inequalities.
By combining all information altogether, we get the coarse Brunner's presentation.
By using the coarse Brunner's presentation, in Section 4 we will give various family of links whose double branched covering has the non-left orderable fundamental group. In Section 5 we briefly discuss a relationship between the results in this paper and the L-space conjectures.\\

\textbf{Acknowledgment}
The author would like to thank Joshua Greene and Liam Watson for comments for eariler version of the paper.
This research was supported by JSPS Research Fellowships for Young Scientists.

\section{Presentation of the fundamental group of double branched cover}
\label{sec:pre}
\subsection{Algebraic and rational tangles}

First we review basic notions of tangles. See \cite{kl} for fundamental facts on rational tangles. For a positive integer $m$, we define the elementary tangles $[\pm m]$ and $[\pm \frac{1}{m}]$ by Figure \ref{fig:tangle}. For two tangles $P$ and $Q$, the {\em tangle sum} $P+Q$ and {\em the tangle product} $P*Q$ are defined as in Figure \ref{fig:tangle}.

\begin{figure}[htbp]
 \begin{center}
\includegraphics[width=100mm]{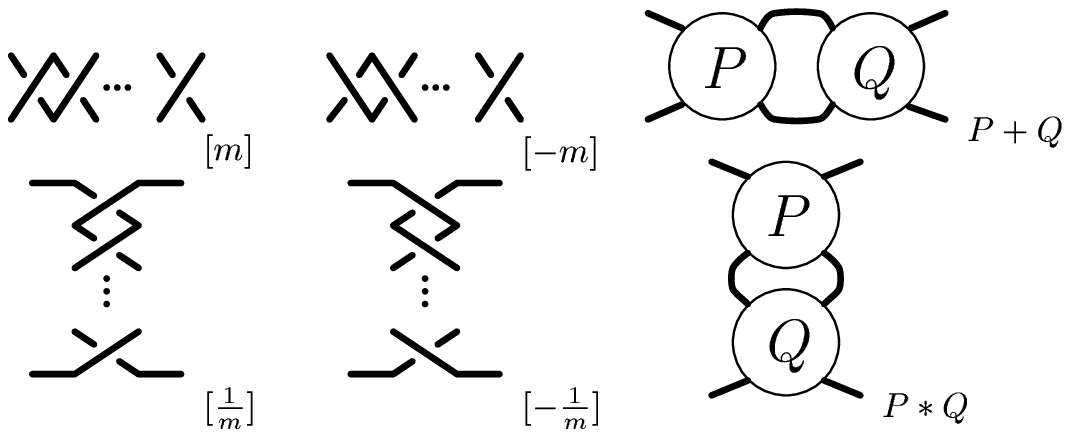}
 \end{center}
 \caption{Elementary tangles and tangle sum, products}
 \label{fig:tangle}
\end{figure}

An {\em algebraic tangle} is a tangle obtained from elementary tangles by repetitions of $+$ and $*$ operations. 

For a non-zero rational number $\frac{q}{p}$, let us take its continued fraction 
\[ \frac{q}{p} = a_{1} + \cfrac{1}{a_{2}+ \cfrac{1}{a_{3} + \cdots}}. \]  

A {\em rational tangle} $Q(\frac{q}{p})$ is an algebraic tangle defined by
\[ Q\left(\frac{q}{p}\right) = 
\left( \cdots \left(
\left([a_{n}]*\frac{1}{[a_{n-1}]}\right)+[a_{n-2}] \right) * \cdots * \frac{1}{[a_{2}]}\right) +[a_{1}]
\]

The isotopy class of the rational tangle $Q(\frac{q}{p})$ does not depend on the choice of the continued fractions. In particular, for a non-zero integer $m$, the rational tangle $Q(m)$ and $Q(\frac{1}{m})$ are isomorphic to the elementary tangle $[m]$ and $[\frac{1}{m}]$.

\subsection{Brunner's presentation}

In this section we review the Brunner's presentation of the fundamental group of double-branched coverings.
Let $D$ be a diagram of a link $L$ in $S^{3}$ and consider the checker board coloring of $D$. We always choose the checker board coloring so that the color of the unbounded region is white. Then regions colored by black defines a (possibly non-orientable) compact surface which bounds $L$. We call this surface the {\em checker board surface}. 

The checker board surface is decomposed as a union of discs and twisted bands.
We choose the maximal disc-twisted band decomposition: the disc-twisted band decomposition having the minimal number of twisted bands. For such a disc-twisted band decomposition of the checker board surface, we associate the labeled planer graph $G$ which we call the {\em decomposition graph}, and the oriented planer graph $\widetilde{G}$ which we call the {\em connectivity graph}.

The vertices of the decomposition graph $G$ are the discs in the maximal disc-twisted band decomposition. For a twisted band connecting two discs, we assign an edge of $G$ so that it connects corresponding vertices. The label of edges are non-zero integer $i(e)$ determined by the signed number of the twisting, as shown in Figure \ref{fig:graph}. We call a component of $\R^{2}- G$ a {\em region} of the link diagram $D$. By definition, a region of $D$ is identified with a white-colored region of the diagram $D$.

\begin{figure}[htbp]
 \begin{center}
\includegraphics[width=100mm]{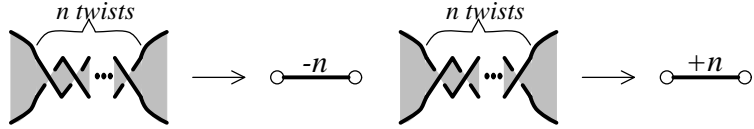}
 \end{center}
 \caption{Labeled edge of the graph $G$}
 \label{fig:graph}
\end{figure}

The connectivity graph $\widetilde{G}$ is obtained from $G$ as follows.
The vertices of $\widetilde{G}$ is the same the vertices of $G$. We connect two vertices of $\widetilde{G}$ by one edge if and only if there exists a twisted band connecting two discs which correspond to the vertices. We choose an arbitrary orientation of edges and make $\widetilde{G}$ as an oriented graph.
We say an edge $w$ of $G$ {\em corresponds} to the edge $W$ of $\widetilde{G}$ if $w$ and $W$ connects the same vertices. 
The orientation of edges in $\widetilde{G}$ induces an orientation of corresponding edges of $G$ so we will always regard $G$ as oriented graph once we choose the connectivity graph $\widetilde{G}$.

Now we are ready to give Brunner's presentation of $\pi_{1}(\Sigma_{2}(L))$.

\begin{thm}[Brunner's presentation of $\pi_{1}(\Sigma_{2}(L))$ \cite{br}]
Let $L$ be an unsplittable link in $S^{3}$ represented by a diagram $D$, and let $G$, $\widetilde{G}$ be the decomposition graph and the connectivity graph.
Then the fundamental group of $\Sigma_{2}(L)$ has the following presentation.\\

\noindent
{\bf [Generators]}
\begin{description}
\item[Edge generators] $\{W_{i}\}$, the set of edges of $\widetilde{G}$.
\item[Region generators] $\{R_{i}\}$, the set of regions of the link diagram $D$.
\end{description}
\noindent
{\bf [Relations]}
\begin{description}
\item[Local edge relations] $W=(R_{l}^{-1}R_{r})^{a}$. Here $R_{l}$, $R_{r}$ are left- and right- adjacent regions of an edge $w$ of $G$ which corresponds to the edge $W$ of $\widetilde{G}$. The exponent $a$ is the label of the edge $w$.
\item[Global cycle relations] $W_{n}^{\pm1} \cdots W_{1}^{\pm 1}=1$ if the edge-path $W_{n}^{\pm 1}\cdots W_{1}^{\pm 1}$ forms a loop in $\R^{2}$. Here $W_{i}^{-1}$ means the edge $W_{i}$ with the opposite orientation.
\item[Vanishing relation] $R_{0}=1$ where $R_{0}$ is the unbounded region.
\end{description}
\end{thm}

Here we use the convention that $W_{2}W_{1}$ represents the edge path which goes along $W_{1}$ first, then goes along $W_{2}$.
We close this section by giving an example of Brunner's presentation.

\begin{exam}
\label{exam:941}
Let us consider the knot diagram $D$ and consider the decomposition graph $G$ and the connectivity graph $\widetilde{G}$ as in Figure \ref{fig:exam}. 
Then by Brunner's theorem, the presentation of $\pi_{1}(\Sigma_{2}(L))$ is given by
\[ \left\langle \begin{array}{c} A,B,C\\ W_{1},\ldots, W_{6}\end{array} 
\begin{array}{|lll}
W_{1}= A^{2} & W_{2}= B^{2} & W_{3}= C \\
W_{4}=(B^{-1}A)^{2} & W_{5}=(B^{-1}C)^{-1} & W_{6}=(C^{-1}A)^{-1} \\
W_{6}W_{4}W_{1}=1 & W_{4}^{-1}W_{5}^{-1}W_{2}=1 & W_{6}^{-1}W_{3}^{-1}W_{5}=1
\end{array} \right\rangle \]
Here we used the vanishing relation to remove the trivial region generator.
\begin{figure}[htbp]
 \begin{center}
\includegraphics[width=110mm]{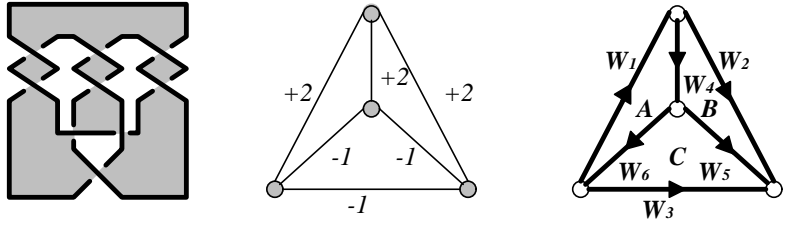}
 \end{center}
 \caption{Example of Brunner's presentation}
 \label{fig:exam}
\end{figure}
\end{exam}

\section{Coarse Brunner's presentation}

In this section we construct the coarse Brunner's presentation by modifying Brunner's presentation.

 \subsection{Reduced Brunner's presentation}

As the first step, we modify the Brunner's presentation to obtain a new group presentation having a simpler form.

Let $D$ be a link diagram and let $A$ be an algebraic tangle which is a part of the diagram $D$. We define the east, west, south and north side of $A$ as in Figure \ref{fig:tancol}. We say $A$ is {\em compatible} with the checker board coloring if for the checkerboard coloring of the diagram $D$, the west- and the east- side of $A$ is colored by black.
We always regard $A$ as oriented from the east side to the west side, hence we call a region generator of $\pi_{1}(\Sigma_{2}(D))$ which corresponds to the north (resp. south) region of $A$ the {\em left-adjacent region} (resp. the {\em right-adjacent region}) of $A$. See Figure \ref{fig:tancol}.
We denote the left- and right- adjacent regions of $A$ by $R_{l}(A)$, $R_{r}(A)$, respectively. 

\begin{figure}[htbp]
 \begin{center}
\includegraphics[width=70mm]{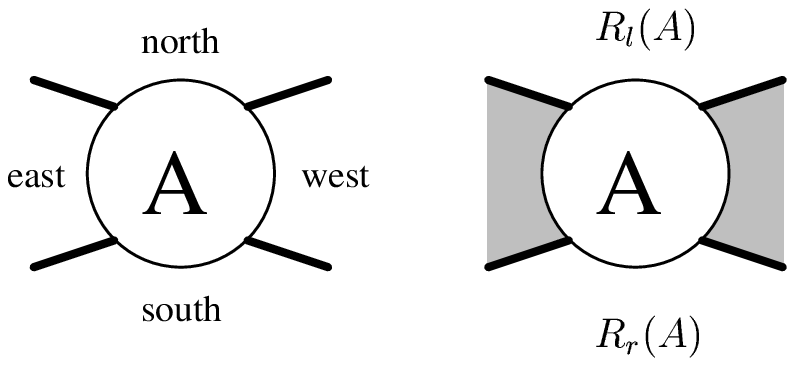}
 \end{center}
 \caption{Adjacent regions and compatible checker board coloring}
 \label{fig:tancol}
\end{figure}

Now we consider the case $A$ is a rational tangle $Q=Q(\pm \frac{1}{m})$.
Then the subgraph of the decomposition graph $G$ derived from sub-diagram $Q$ is $m$ parallel edges connecting the same vertices. All edges have the same label $\pm 1$. We will represents such a sub-graph by one edge with label $\pm \frac{1}{m}$, as shown in Figure \ref{fig:reduced}. 
We call the graph obtained by such a replacement {\em the reduced decomposition graph} and denote by $\red(G)$.

\begin{figure}[htbp]
 \begin{center}
\includegraphics[width=110mm]{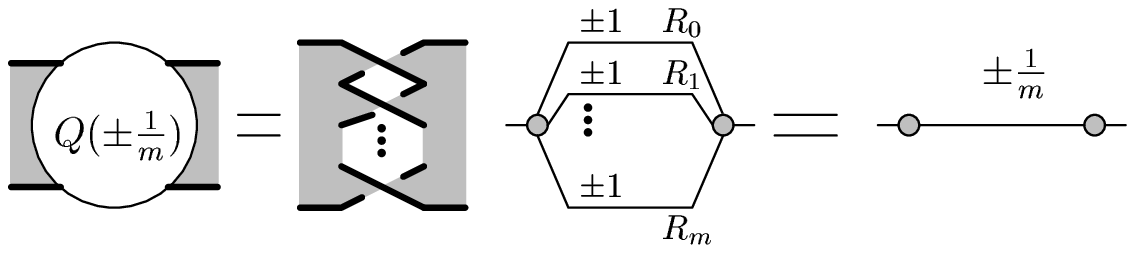}
 \end{center}
 \caption{Reduced decomposition graph}
 \label{fig:reduced}
\end{figure}

The local edge relations derived in the subgraph are given by
\[ W = (R_{0}^{-1}R_{1})^{\pm 1}= (R_{1}^{-1}R_{2})^{\pm 1} = \cdots = (R_{m-1}^{-1}R_{m})^{\pm 1}.\]
where $R_{0},\ldots,R_{m}$ are region generators taken as in Figure \ref{fig:reduced}. 
Observe that $R_{0}=R_{l}=R_{l}(Q)$, the left-adjacent region of $Q$ and $R_{m}=R_{r}=R_{r}(Q)$, the right-adjacent region of $Q$.

From these relations, we obtain a simple relation
\[ W^{m}= (R_{0}^{-1}R_{m})^{\pm 1} = (R_{l}^{-1}R_{r})^{\pm 1} \]
We call this relation the {\em reduced local edge relation} derived from the edge labeled by $\pm \frac{1}{m}$.

We interpret the reduced local edge relation as a generalization of the local edge relations in Brunner's presentation. Recall that in the reduced decomposition graph, the tangle $Q(\pm \frac{1}{m})$ is described as an edge $e$ labeled with $\pm \frac{1}{m}$.
Thus by adapting the local edge relation for the edge $e$ formally, we get a ``local edge relation" 
\[ W = (R_{l}^{-1}R_{r}) ^{\pm \frac{1}{m}}.\]
Of course, this equation is meaningless, but by taking the $m$-th power we get an actual relation of elements in $\pi_{1}(\Sigma_{2}(L))$,
\[ W^{m}= (R_{0}^{-1}R_{m})^{\pm 1}. \]

We remove region generators $R_{1},\ldots,R_{m-1}$ from Brunner's presentation and replace the local edge relations 
\[ W = (R_{0}^{-1}R_{1})^{\pm 1}= (R_{1}^{-1}R_{2})^{\pm 1} = \cdots = (R_{m-1}^{-1}R_{m})^{\pm 1}.\] 
with the reduced local edge relation 
\[ W^{m}= (R_{0}^{-1}R_{m})^{\pm 1} \]
We say the obtained group presentation the {\em reduced Brunner's presentation}. 

As in Brunner's presentation, the reduced Brunner's presentation is obtained from the reduced decomposition graph $\red(G)$ and the connectivity graph $\widetilde{G}$ in a same manner.
 
We denote the group defined by the reduced Brunner's presentation by $\pi(D)$.
We remark that the group $\pi(D)$ given by reduced Brunner's presentation is not the same as the $\pi_{1}(\Sigma_{2}(L))$, since we have lost information for generators $R_{1},\ldots,R_{m-1}$. 

\begin{lem}
\label{lem:non-trivial}
$\pi(D)$ is a non-trivial group.
\end{lem}
\begin{proof}
If $\pi(D)$ is trivial, then all region generators in the reduced Brunner's presentation of $\pi(D)$ represents the trivial element. Since the relations in the reduced Brunner's presentation are deduced from Brunner's presentation, this implies that all region generators in Brunner's presentation which appears in reduced Brunner's presentation also represent the trivial element of $\pi_{1}(\Sigma_{2}(L)$. 

Then in turn, by the local edge relations in Brunner's presentation we conclude all other generators in Brunner's presentation also represent the trivial element in $\pi_{1}(\Sigma_{2}(L)$. Thus $\pi_{1}(\Sigma_{2}(L)$ is a trivial group.
This cannot happen since we are considering the double branched covering of links in $S^{3}$.
\end{proof}

\subsection{Tangle elements}

Next we define the tangle element $W_{A}$ which will appear as a generator of the coarse presentation.

Let $A$ be an algebraic tangle in a link diagram $D$ which is compatible with the checker board coloring. We consider the subgraph of the reduced decomposition graph $\red(G)$ and $\widetilde{G}$ which corresponds to the sub-diagram $A$. For each edges in the subgraph of $\widetilde{G}$, we assign an orientation from left to right, in other words, from the west region of $A$ to the east region $A$.

Let us take an oriented simple edge-path of $\widetilde{G}$ which start from the leftmost vertex, the west region of $A$, and end at the rightmost vertex, the east region of $A$.
We call the element of $\pi(D)$ defined by such an oriented edge-path the {\em tangle element} of $A$, and denoted by $W_{A}$. By the global cycle relations $W_{A}$ does not depend on a choice of the edge-path.

We study fundamental properties of a tangle element. We begin with the rational tangle case.
Let \[ Q \left(\frac{q}{p} \right) = 
\left( \cdots \left(
\left([a_{n}]*\frac{1}{[a_{n-1}]}\right)+[a_{n-2}] \right)* \cdots * \frac{1}{[a_{2}]}\right) +[a_{1}]
\]
be a rational tangle. The subgraph of the reduced decomposition graph and the connectivity graph derived from $Q$ are given in Figure \ref{fig:rationalgraph}. Let us take region generators and edge generators as in Figure \ref{fig:rationalgraph}. Thus, $R_{0}=R_{r}(Q)$ and $R_{n}=R_{l}(Q)$.
By definition, the tangle element $W_{Q}$ is given by
\[ W_{Q} = W_{1}W_{3}\cdots W_{2n-1}W_{2n}. \]

We remark that the two regions $R_{0}$ and $R_{n}$ might be the same region if we consider the whole diagram $D$, but at this moment we regard them as different regions.

\begin{figure}[htbp]
 \begin{center}
\includegraphics[width=120mm]{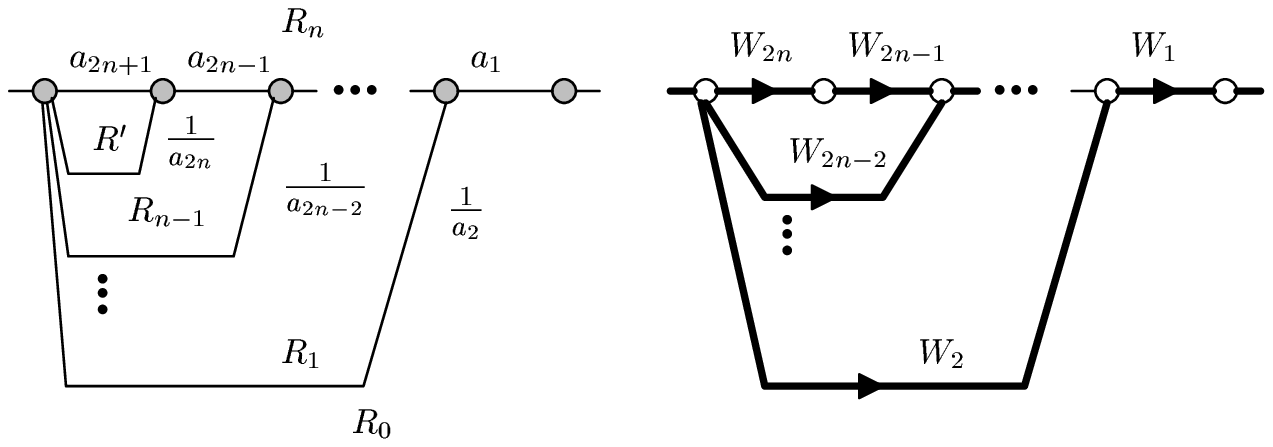}
 \end{center}
 \caption{Subgraphs derived from the rational tangle $Q(q\slash p)$}
 \label{fig:rationalgraph}
\end{figure}

The set of reduced local edge relations the global cycle relations derived from the regions and edges in the subgraph are give as follows:
\[ \left\langle 
\begin{array}{c}
W_{1},\ldots ,W_{2n}\\
R_{0},\ldots ,R_{n},R'
\end{array}
\left| 
\begin{array}{ll}
W_{2i}^{a_{2i}}=(R_{i}^{-1}R_{i-1}) & 1 \leq i \leq n-1 \\
W_{2i+1} = (R_{k}^{-1}R_{i})^{a_{2i+1}} & 0 \leq i \leq n-1\\
W_{2n}^{a_{2n}}=(R'^{-1}R_{n-1}) & W_{2n}=(R_{n}^{-1}R')^{a_{2n+1}} \\
W_{2i}=W_{2i+1}W_{2i+2} & 1\leq i \leq n-1
\end{array}
\right.\right\rangle
\]

\begin{lem}
\label{lem:rt}
For $1\leq i \leq n$, let $q_{i} \slash p_{i}$ be a rational number defined by the continued fraction
\[ \frac{q_{i}}{p_{i}} = \cfrac{1}{a_{2i}+ \cfrac{1}{a_{2i+1} + \cdots}}. \]  
Then $W_{2i}$ commutes with $R_{n}^{-1}R_{i-1}$, and $W_{2i}^{p_{i}} = (R_{n}^{-1}R_{i-1})^{q_{i}}$.
\end{lem}

\begin{proof}
We prove lemma by induction on $i$.

First we consider the case $i=n$.
By local edge relations $W_{2n}^{a_{2n}}=(R'^{-1}R_{n-1})$ and $W_{2n}=(R_{n}^{-1}R')^{a_{2n+1}}$, 
$W_{2n}$ commutes with both $(R'^{-1}R_{n-1})$ and $(R_{n}^{-1}R')$.
Thus $W_{2n}$ commutes with $(R_{n}^{-1}R_{n-1})= (R_{n}^{-1}R')(R'^{-1}R_{n-1})$.
Moreover, $(R'^{-1}R_{n-1})$ commutes with $(R_{n}^{-1}R')$. By definition, $p_{n}=a_{2n+1}a_{2n}+1$ and $q=a_{2n+1}$, therefore we conclude
\[ W_{2n}^{p_{n}}= W_{2n}^{a_{2n} \cdot a_{2n+1}} W_{2n} = (R_{n}^{-1}R')^{a_{2n+1}}(R'^{-1}R_{n-1})^{a_{2n+1}} = (R_{n}^{-1}R_{n-1})^{q_{n}}.\]

Assume that we have proved the Lemma for $>i$.
By local edge relations, $W_{2i}$ commutes with $(R_{i}^{-1}R_{i-1})$ and $W_{2i+1}$ commutes with $(A_{n}^{-1}A_{i})$.
By inductive hypothesis, $W_{2i+2}$ commutes with $(R_{n}^{-1}R_{i})$.
Since $W_{2i}=W_{2i+1}W_{2i+2}$, $W_{2i}$ commutes with $R_{n}^{-1}R_{i}$ and $(R_{n}^{-1}R_{i-1})=(R_{n}^{-1}R_{i})(R_{i}^{-1}R_{i-1})$.

In particular, $W_{2i}$ commutes with $W_{2i+1}=(R_{n}^{-1}R_{i})^{a_{2i+1}}$.
Thus
\[ W_{2i}^{p_{i+1}} = W_{2i+2}^{p_{i+1}} W_{2i+1}^{p_{i+1}} = (R_{n}^{-1}R_{i})^{q_{i+1} + a_{2i+1}p_{i+1}}.\]
so we conclude
\[
 W_{2i}^{p_{i}} = W_{2i}^{p_{i+1} + a_{2i}(q_{i+1} + a_{2i+1}p_{i+1})} = (R_{n}^{-1}R_{i})^{q_{i+1} + a_{2i+1}p_{i+1}} =  (R_{n}^{-1}R_{i})^{q_{i}}.
\]
\end{proof}

Lemma \ref{lem:rt} shows the following properties of $W_{A}$ for an algebraic tangle $A$.

\begin{prop}
\label{prop:rational}
Let $A$ be an algebraic tangle in a link diagram $D$ which is compatible with the checker board coloring, and let $R_{l}=R_{l}(A)$ and $R_{r}=R_{r}(A)$ be the left- and right- adjacent regions of $A$.
Then as an element of $\pi(D)$, the tangle element $W_{A}$ has the following properties. 
\begin{enumerate}
\item If $A$ is a rational tangle $Q(\frac{q}{p})$, then $W_{A}^{p}=(R_{l}^{-1}R_{r})^{q}$.
\item $W_{A}$ commutes with $(R_{l}^{-1}R_{r})$.
\end{enumerate}
\end{prop}
\begin{proof}
(1) is obvious from Lemma \ref{lem:rt}.
Assume that $A$ is obtained from $n$ rational tangles by applying $+$ or $*$ operations.
We prove (2) by induction on $n$. The case $n=1$ is proved in Lemma \ref{lem:rt}.

Assume that $A=A' + Q$ where $A'$ is an algebraic tangle obtained from $n-1$ rational tangles, and $Q$ be a rational tangle. Then $R_{l}(A)=R_{l}(A')=R_{l}(Q)$ and $R_{r}(A)= R_{r}(A') = R_{r}(Q)$, hence by induction $A=A'+Q$ commute with $(R_{l}^{-1}R_{r})$. Similarly, assume that $A=A'*Q$. Then $R_{l}(A)=R_{l}(A')$, $R_{r}(A') = R_{l}(Q)$, and $R_{r}(Q) = R_{r}(A)$. By the global cycle relation, $W_{A}=W_{A'}=W_{Q}$. Hence $W_{A}$ commutes with $(R_{l}^{-1}R_{r}) =(R_{l}(A')^{-1}R_{r}(A'))(R_{l}(Q)^{-1}R_{r}(Q))$. 
\end{proof}

\subsection{Universal range}

In this section we present a notion of universal range. This is an inequality of group elements which is valid for all left orderings.

For a left ordering $<_{G}$ of $G$ and rational numbers $a=\frac{q}{p}, b=\frac{s}{r}$ such that $a<b$, let $[[a,b]]_{X,<_{G}}$ be a subset of $G$ defined by
\[ [[a,b]]_{X, <_{G}} = \left\{
\begin{array}{ll}
 \{ g \in G \: | \: X^{q} \leq_{G} g^{p},  g^{r} \leq_{G} X^{s}, Xg=gX \} & ( \textrm{if } X >_{G} 1) \\
 \{ g \in G \: | \: X^{s} \leq_{G} g^{r}, g^{p} \leq_{G} X^{p}, Xg=gX\} & ( \textrm{if } X <_{G} 1).
\end{array}
\right.
 \]
We also define
\begin{align*}
 [[ a , +\infty ]]_{X, <_{G}} = \bigcup_{b>a} [[a,b]]_{X, <_{G}}, \\
  [[ -\infty , b ]]_{X, <_{G}} = \bigcup_{b>a} [[a,b]]_{X, <_{G}}.
\end{align*}

If $g$ commutes with $X$, then $X^{mp} < g^{mq}$ if and only if $X^{p} < g^{q}$. Hence the subset $[[a,b]]_{X,<_{G}}$ does not depend on a choice of the representatives of rationals $a= \frac{p}{q}$ and $b= \frac{r}{s}$.
 
Now we define $[[a,b]]_{X}$ by 
\[ [[a,b]]_{X} = \bigcap_{<_{G} \in LO(G)} [[a,b]]_{X,<_{G}} \]
where $LO(G)$ denotes the set of all left orderings of $G$.

\begin{lem}
\label{lem:univ}
Let $X$ and $Y$ be non-trivial elements of $G$ and $a,b,c,d \in \Q \cup \{ \pm \infty \}$ be rational numbers such that $a\leq b$, $c \leq d$. 
\begin{enumerate}
\item  If $[a,b] \subset [c,d]$  as a subset of $\Q$, then $[[a,b]]_{X} \subset [[c,d]]_{X} $ as a subset of $G$. 
\item  Assume that $g \in [[a,b]]_{X}$ and $h \in [[c,d]]_{X}$.
If $gh=hg$ or $a,b,c,d \in \Z$, then $gh \in [[a+c,b+d]]_{X}$.
\item  Assume that $g \in [[a,b]]_{X}$ and $g \in [[c,d]]_{Y}$.
 If $XY=YX$ or $a,b,c,d \in 1 \slash \Z =\{ \frac{1}{n} \: | \: n \in \Z \cup \{\pm 0\} \}$, then $g \in [m,M]_{XY}$, where $m$ and $M$ are defined by
\[ \left\{\begin{array}{l}
m = \min \{  (a^{-1}+c^{-1})^{-1},(a^{-1}+d^{-1})^{-1},(b^{-1}+c^{-1})^{-1}, (b^{-1}+d^{-1})^{-1}\} \\
M = \max \{  (a^{-1}+c^{-1})^{-1},(a^{-1}+d^{-1})^{-1},(b^{-1}+c^{-1})^{-1}, (b^{-1}+d^{-1})^{-1}\}.
\end{array}\right.
\]
Here we regard $+\infty$ as $\frac{1}{+0}$ and $-\infty$ as $\frac{1}{-0}$.
\end{enumerate}
\end{lem}

\begin{proof}
(1) is obvious from the definition.
We prove (2) and (3). Let us put $a= p \slash q$ and $c= r \slash s$, where $p,q,r,s \in \Z$.
 Let $<_{G}$ be a left ordering of $G$. We only consider the case that $1<_{G} X$ and $1<_{G}Y$. We show the lower bounds $X^{ps+qr} \leq_{G} (gh)^{pr}$ and  $(XY)^{sq} \leq_{G} g^{sp +qr}$ respectively. The other cases and the upper bounds are proved in a similar way.
 
By assumption, $X^{q} \leq_{G} g^{p}$ and $X^{s} \leq_{G} h^{r}$. Since $X$ commutes with both $g$ and $h$, 
$X^{rq}  \leq_{G} g^{rp}$ and $X^{ps}  \leq_{G} h^{pr}$.  If $g$ and $h$ commutes or both $a$ and $c$ are integers, then $X^{ps+qr}  \leq_{G} (gh)^{pr}$. 
Similarly, since $g$ commutes with both $X$ and $Y$, the assumption $X^{q}  \leq_{G} g^{p}$ and $Y^{s}  \leq_{G} g^{r}$ implies $X^{qs}  \leq_{G} g^{sp}$ and $Y^{sq}  \leq_{G}g^{qr}$.
If $X$ and $Y$ commutes or both $a, c \in 1 \slash \Z$, then $(XY)^{sq}  \leq_{G} g^{sp+qr}$. 
\end{proof}

\begin{defn}
Let $D$ be a link diagram representing a link $L$ and $A$ be an algebraic tangle in $D$ which is compatible with the checkerboard coloring. Let $R_{l}=R_{l}(A),R_{r}=R_{r}(A)$ be the left- and the right- adjacent regions of $A$ and $W_{A}$ be the tangle element, defined in Section 3.2.
If $W_{A} \in [[m,n]]_{R_{l}^{-1}R_{r}}$, we say the interval $[[m,n]]$ is a {\em universal range} of $A$, and denote by $A \in [[m,n]]$.
\end{defn}

If $[[m,n]]$ is a universal range of $A$ and $m'<m$, $n<n'$, then $[[m',n']]$ is also a universal range of $A$.
In general it is difficult to determine the optimal universal range, the smallest universal range which will depends on not only $A$ but also the global diagram $D$.
However, we can easily obtain non-trivial universal range of $A$ by using only the structure of $A$ itself.

First observe that if $A$ is a rational tangle $Q(\frac{q}{p})$, then by Proposition \ref{prop:rational},$Q(\frac{q}{p}) \in [[\frac{q}{p},\frac{q}{p}]]$.
The universal range for general algebraic tangles can be obtained from the next Proposition.

\begin{prop}
\label{prop:algebraic}
Let $D$ be a link diagram and $A$, $A_{1}$ and $A_{2}$ be algebraic tangles, which are sub-diagram of $D$ and compatible to the checker board coloring.  Assume that $A_{1} \in [[a,b]]$, and $A_{2} \in [[c,d]]$, where $a,b,c,d \in \Q \cup \{ \pm \infty\}$.
\begin{enumerate}
\item  Assume that $A=A_{1}+A_{2}$. 
If $a,b,c,d \in \Z$, or $A_{1}$ and $A_{2}$ commute, then $A \in[[a+c,b+d]]$.
\item  Assume that $A=A_{1}\,*\, A_{2}$. 
If $a,b,c,d \in 1 \slash \Z \cup \{\pm 0\}$, or $R_{l}(A_{1})^{-1}R_{r}(A_{1})$ and $R_{l}(A_{2})^{-1}R_{r}(A_{2})$ commute, then $A \in [[m,M]]$, where 
\[ \left\{\begin{array}{l}
m = \min \{  (a^{-1}+c^{-1})^{-1},(a^{-1}+d^{-1})^{-1},(b^{-1}+c^{-1})^{-1}, (b^{-1}+d^{-1})^{-1}\} \\
M=\max \{ (a^{-1}+c^{-1})^{-1},(a^{-1}+d^{-1})^{-1},(b^{-1}+c^{-1})^{-1}, (b^{-1}+d^{-1})^{-1}\}
\end{array}
\right.
\]
Here we regard $+\infty$ as $\frac{1}{+0}$ and $-\infty$ as $\frac{1}{-0}$.

\end{enumerate}
\end{prop}
\begin{proof}
These assertions follow by Proposition \ref{prop:rational} and Lemma \ref{lem:univ}.  
\end{proof}

We here remark that the universal range obtained from Proposition \ref{prop:algebraic} is far from optimal. In most cases, a non-optimal universal range computed by Proposition \ref{prop:algebraic} is sufficient to apply our Theorems described in next section. With some additional assumptions and more careful arguments, we will often be able to  get better universal range. We give an example how to get better universal range.

\begin{exam}
\label{exam:better}
Let us consider the algebraic tangle $A = A_{1} + A_{2} = Q(\frac{1}{3}) + Q(\frac{1}{4})$. According to Proposition \ref{prop:algebraic}, one can directly get $A \in [[0,2]]$, since $Q(\frac{1}{3}) \in [[0,1]]$ and $Q(\frac{1}{4}) \in 
[[0,1]]$. 
However, in this case we can get better universal range as follows. 

Let $X = (R_{l}(A)^{-1}R_{r}(A))$,  $W_{1}=W_{A_{1}}$ and $W_{2}=W_{A_{2}}$.
Then $W_{1}^{3} = W_{2}^{4} = X$. 
Let $<$ be a left-ordering. We consider the case $1<X$. The case $1> X $ is treated in a similar way. Now
\[ W_{A}^{3}=(W_{1}W_{2})^{3} = W_{1}^{3}(W_{1}^{-2}W_{2}W_{1}^{2}) (W_{1}^{-1}W_{2}W_{1})W_{2} \]
Since $(W_{1}^{-i}W_{2}W_{1}^{i})^{4} = X$, $W_{1}^{-i}W_{2}W_{1}^{i}>1$ holds for all $i$. Thus, we get
\[ W_{A}^{3} = (W_{1}W_{2})^{3} > W_{1}^{3} = X. \]
Similarly, we have
\[ W_{A}^{4}= (W_{1}W_{2})^{4} = W_{1} W_{2}^{4} (W_{2}^{-3} W_{1}W_{2}^{3}) (W_{2}^{-2} W_{1}W_{2}^{2}) (W_{2}^{-1}W_{1}W_{2}) \]
Since $(W_{2}^{-i}W_{1}W_{2}^{i})^{3} = X$, $W_{2}^{-i}W_{1}W_{2}^{i} < X $ for all $i$. Thus, we get
\[  W_{A}^{4}= (W_{1}W_{2})^{4}  < X^{5}. \]

Thus, we get a better universal range $ A \in [[\frac{1}{3},\frac{5}{4}]]$.
\end{exam}

By Proposition \ref{prop:algebraic}, if $A$ is an alternating algebraic tangle, then $A \in [[0,+\infty]]$ or $A \in [[-\infty,0]]$. The converse is not true.

\begin{exam}
\label{exam:univrange}

Let $A= [((Q(\frac{1}{3}) +Q( \frac{1}{4}))*Q(-1)] + Q(2)$ be a non-alternating algebraic tangle.
As we have seen in Example \ref{exam:better}, $(Q(\frac{1}{3}) +Q( \frac{1}{4})) \in [[\frac{1}{3},\frac{5}{4}]] \subset [[\frac{1}{3}, +\infty = \frac{1}{+0}]]$.
Thus, by Proposition \ref{prop:algebraic}, $((Q(\frac{1}{3}) +Q( \frac{1}{4}))*Q(-1) + Q(2) \in [[1,\frac{5}{2}]]$.
\end{exam}

\subsection{Coarse Brunner's presentation}

Now we are ready to give the coarse Brunner's presentation.

Let $D$ be a link diagram representing $L$. We consider a decomposition of $D$ as a union of algebraic tangles and strands so that all crossings of $D$ are contained in some tangle parts, and that each algebraic tangle is compatible with the checker board colorings. 
We say such a decomposition of link diagram a {\em tangle-strand decomposition}.

A tangle-strand decomposition of $D$ defines the decomposition of the checker board surface of $D$ as discs and subsurfaces corresponding to tangles. Then we construct the oriented labeled graph $\Gamma$, which we call the {\em coarse decomposition graph} as follows.
The vertex of $\Gamma$ is a disc part of the tangle-strand decomposition. For each algebraic tangle $A$, we assign an oriented edge of $\Gamma$ having the label $A$ as in Figure \ref{fig:coarse}. 
We remark that if we choose all tangles $A$ as integer tangles, then $\Gamma$ is nothing but a decomposition graph, and if we choose all tangles $A$ as integer tangles or the rational tangles of the form $Q(\pm \frac{1}{n})$, then $\Gamma$ is a reduced decomposition graph. Thus the coarse decomposition graph $\Gamma$ is a generalization of the (reduced) decomposition graph.

\begin{figure}[htbp]
 \begin{center}
\includegraphics[width=70mm]{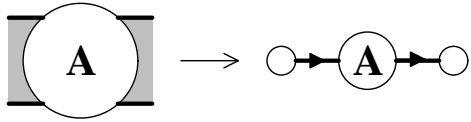}
 \end{center}
 \caption{The coarse decomposition graph $\Gamma$}
 \label{fig:coarse}
\end{figure}

Using the coarse decomposition graph, we define the coarse Brunner's presentation.

\begin{defn}
The coarse Brunner's presentation $\CB$ for a tangle-strand decomposition of a link diagram $D$ is a set of generators and relations given as follows.\\

{\bf [Generators]}
\begin{description}
\item[Tangle generator] $\{W_{A}\}$, the set of edges of $\Gamma$.
\item[Region generator] $\{R_{i}\}$, the set of connected components of $\R^{2}-\Gamma$.
\end{description}

{\bf [Relations]}
\begin{description}
\item[Local Coarse relation] $W_{A} \in [[m_{A}, M_{A}]]_{(R_{l}^{-1}R_{r})}$, where $R_{l}=R_{l}(A)$ and $R_{r}=R_{r}(A)$ are the left- and right- adjacent regions of $A$ and $[[m_{A},M_{A}]]$ is the universal range of $A$. 

\item[Global cycle relation] $W_{n}^{\pm1} \cdots W_{1}^{\pm 1}=1$ if the edge-path $W_{n}^{\pm 1}\cdots W_{1}^{\pm 1}$ forms a loop in $\R^{2}$. Here $W_{i}^{-1}$ represents the path $W_{i}$ with the opposite orientation.
\item[Vanishing relation] $R_{0}=1$ where $R_{0}$ corresponds to the unbounded region.
\end{description}

\end{defn}

Let us compare Brunner's presentation of  $\pi_{1}(\Sigma_{2}(L))$ with the coarse Brunner's presentation $\CB$. Recall that each tangle element $W_{A}$ is regarded as an element of $\pi_{1}(\Sigma_{2}(L))$.
Conversely, some region generators of $\pi_{1}(\Sigma_{2}(L))$ in Brunner's presentation are naturally regarded as region generators of $\CB$. 

The local coarse relation $W_{A} \in [[m_{A}, M_{A}]]_{(R_{l}^{-1}R_{r})}$ is regarded as two inequalities and the commutativity relation $W_{A}(R_{l}^{-1}R_{r})=(R_{l}^{-1}R_{r})W_{A}$.
As we have seen in the previous sections, the local coarse relations are consequences of relations of the reduced Brunner's presentation. Hence the local coarse relations are the consequence of Brunner's presentation of $\pi_{1}(\Sigma_{2}(L))$. Therefore, the inequalities and the commutative relation in the local coarse relation are valid as relations of elements in $\pi_{1}(\Sigma_{2}(L))$.

We close this section by giving an example of coarse Brunner's presentation.
\begin{exam}
\label{exam:coarse}
Let us consider a link diagram given in Figure \ref{fig:diagram} left, whose coarse decomposition graph $\Gamma$ is given in Figure \ref{fig:diagram} right.
For each algebraic tangle $A_{i}$, let $[[m_{i},M_{i}]]$ be the universal range of $A_{i}$. 

Thus, the coarse Brunner's presentation of $\pi_{1}(\Sigma_{2}(L))$ is given as follows:\\

\noindent 
 {\bf [Generators:]}
 \[  W_{1}, W_{2},\ldots, W_{6},A,B,C \]
{\bf [Local Coarse relations:]}
\[ 
\begin{array}{ll}
W_{1}\in [[m_{1},M_{1}]]_{A} & W_{2} \in [[m_{2},M_{2}]]_{B}\\
W_{3} \in [[m_{3},M_{3}]]_{C^{-1}} & W_{4}\in [[m_{4},M_{4}]]_{B^{-1}A}\\
W_{5}\in [[m_{5},M_{5}]]_{B^{-1}C} & W_{6}\in [[m_{6},M_{6}]]_{C^{-1}A}
\end{array}
\]
{\bf [Global cycle relations:]}
\[ 
W_{6}W_{4}W_{1}=1, \;\;\; W_{4}^{-1}W_{5}^{-1}W_{2}=1,\;\;\; W_{6}^{-1}W_{3}^{-1}W_{5}=1
\]
Here we put $W_{A_{i}}=W_{i}$, and we simplify the presentation by removing the trivial region generator.

\begin{figure}[htbp]
 \begin{center}
\includegraphics[width=80mm]{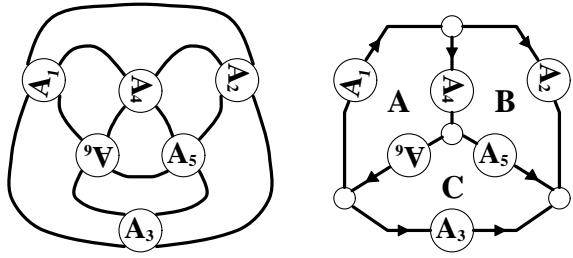}
 \end{center}
 \caption{Example of coarse Brunner's presentation}
 \label{fig:diagram}
\end{figure}

\end{exam}

\section{Non-left-orderable double branched covering}

In this section we use the coarse Brunner's presentation to show that the double branched covering of links represented by certain diagrams has the non-left-orderable fundamental group.
As we already mentioned, the reader might regard the proof of these ``Theorems" rather as examples of how to show non-left-orderability via coarse Brunner's presentations.
The arguments appearing in the following proofs are typical examples to deduce contradiction: We may apply similar argument to deduce contradictions for coarse Brunner's presentation for other link diagrams.

\subsection{Diagrams which are close to alternating diagrams}

The first example we treat is a link which is similar to alternating in the coarse presentation view point.
The proof of the next Theorem is inspired by an argument of Greene in \cite{gr}, and is regarded as an adaptation of Greene's argument for coarse Brunner's presentation. 
 
\begin{thm}
\label{thm:simalter}
Let $D$ be a link diagram which admits a tangle-strand decomposition such that $A_{i} \in [[0,\infty]]$ for all $i$, or $A_{i} \in [[-\infty,0]]$ for all $i$. Then $\pi_{1}(\Sigma_{2}(L))$ is not left-orderable.
\end{thm}
\begin{proof}

Assume that $\pi_{1}(\Sigma_{2}(L))$ has a left ordering $<$.
Let $\CB$ be the coarse Brunner's presentation. We prove the case $A_{i} \in [[0,+\infty]]$ for all $i$. The case all $A_{i} \in [[-\infty,0]]$ is proved in a similar way.
Let $\{R_{1},\ldots,R_{m}\}$ be the set of all region generators of $\CB$, including the trivial region generator which corresponds to the unbounded region, and $\{W_{1},\ldots,W_{n}\}$ be a set of tangle generators of $\CB$, where we put $W_{i}=W_{A_{i}}$.

First of all, we consider the case all region generators represents the same element of $\pi_{1}(\Sigma_{2}(L))$. Since the unbounded region generator represents the trivial element of $\pi_{1}(L)$, this implies that all region generators represent the trivial element. 
In such case, $A_{i} \in [[0,+\infty]]$ implies $W_{i}=1$ for all $i$. Then the group $\pi(D)$ defined by reduced Brunner's presentation is trivial, which contradicts Lemma \ref{lem:non-trivial}.
Thus, we may assume that there are at least two distinct region generators.
Let $R$ be the region generator which is $<$-maximal among the set of all region generators. Since we assumed that there are at least two distinct region generators, we may choose $R$ so that there are region generator $R'$ which is adjacent to $R$ and $R'<R$ holds. 

Let us consider the global cycle relation $W_{1}^{\pm1} \cdots W_{k}^{\pm 1}=1$ given by the edge-path representing the boundary of the region $R$.
Since the property $A_{i} \in [[0,+\infty]]$ is independent of the choice of the orientation of the edge of the coarse decomposition graph $\Gamma$, we may choose the orientation of edges of $\Gamma$ so that the global cycle relation is given as $W_{1} \cdots W_{k}=1$.

For each $i$, let $R_{i}$ be the left-adjacent region of $A_{i}$. Since $R_{i} \leq R$, the local coarse relation $W_{i} \in [[0,+\infty]]_{R_{i}^{-1}R}$ implies $W_{i} \geq 1$. Moreover, as we have assumed,  the inequality must be strict for some $i$: $W_{i}>1$ holds for some $i$. Thus, $W_{1}\cdots W_{n} >1$. This is a contradiction.
\end{proof}

As a corollary, we recover a result of Boyer-Gordon-Watson \cite{bgw}, \cite{gr}.

\begin{cor}
\label{cor:alter}
The fundamental group of the double branched covering of an alternating link is not left-orderable.
\end{cor}
\begin{proof}
Let $D$ be an alternating link diagram representing an alternating link $L$
By regarding each crossing of $D$ as a tangle part the elementary tangle $[\pm 1]$, we get the tangle-strand decomposition of $D$.
Since $D$ is alternating, all tangle parts have the same sign. Thus by Theorem \ref{thm:simalter}, $\pi_{1}(\Sigma_{2}(L))$ is not left-orderable.
\end{proof}

We remark that as we have observed in Example \ref{exam:univrange}, there are non-alternating algebraic tangles whose universal range are either $[[0,+\infty]]$ or $[[-\infty,0]]$. Thus, links in Theorem \ref{thm:simalter} contain a lot of non-alternating links.

\subsection{Various families of non-left-orderable double branched covering.}
\label{sec:proof}

Next we give other family of links whose double branched cover has non-left-orderable fundamental groups.
These links are not similar to alternating in the sense Theorem \ref{thm:simalter}. These examples are derived from certain quasi-alternating diagrams. See Remark \ref{rem:diagram} given in Section 5.

Before proving the non-left-orderability, we observe the following rather obvious fact.
\begin{lem}
\label{lem:conjugate}
Let $G$ be a group and $A,X,Y \in G$. Assume that $X^{p} = A^{q}$ for some positive integers $p$ and $q$. Then for a left-ordering $<$ of $G$, if $1 \leq Y^{-1}AY$ (resp. $1 \geq Y^{-1}AY$), then $1 \leq Y^{-1}X Y$ (resp. $1 \geq Y^{-1}XY$).
\end{lem}
\begin{proof}
By hypothesis,
\[ 1 \leq (Y^{-1}AY) \leq (Y^{-1}AY)^{q} = Y^{-1}X^{p}Y = (Y^{-1}XY)^{p} \]
Thus, $ 1 \leq Y^{-1}XY$. 
\end{proof}

\begin{thm}
\label{thm:main}
Let $L$ be a link in $S^{3}$ which is represented by a diagram Figure \ref{fig:diagram} given in Example \ref{exam:coarse}, and let $[[m_{i},M_{i}]]$ be the universal range of algebraic tangles $\{A_{i}\}$.
Assume that one of the following conditions.
\begin{enumerate}
\item $m_{1},m_{2},m_{4} \geq 1$, $-1 \leq m_{3}, m_{5}, m_{6}$, and  $M_{3},M_{5},M_{6} < 0$
\item $m_{1},m_{2},m_{3},m_{4},m_{5} \geq 1$, and $A_{6} = Q(r)$ where $-1 \leq r < 0$.
\end{enumerate}
Then the fundamental group of the double branched cover $\Sigma_{2}(L)$ is not left-orderable.
\end{thm}

\begin{proof}

Let us consider the coarse Brunner's presentation obtained from link diagram $D$, which we have already given in Example \ref{exam:coarse}. In the proof of theorem, we frequently use the relation $W_{2}W_{1}=W_{3}$ deduced from the global cycle relations. 
Recall that by Lemma \ref{lem:non-trivial}, the group obtained by the reduced Brunner's presentation $\pi(D)$ is not the trivial group.
Thus, at least one of the region generators in the coarse Brunner's presentation must be non-trivial.
Assume that $\pi_{1}(\Sigma_{2}(L))$ has a left-ordering $<$.

First we consider the case that the assumption (1) holds.
With no loss of generality, we can assume $A \leq B$. \\

\noindent
{\bf Case 1: $A \leq B \leq C$}\\

In this case by the universal ranges of $W_{4}$, $W_{5}$ and $W_{6}$ we get 
$W_{4} \leq 1$, $W_{5} \leq 1$ and $W_{6} \geq 1$. So $W_{3}=W_{5}W_{6}^{-1}  \leq 1$.
Since $W_{3} \in [[-1, M_{3}]]_{C^{-1}}$, $(-1 \leq M_{3}<0)$, $C \leq 1$.
Thus, we have $A \leq B \leq  C \leq 1$. Since $W_{1} \in [[+1,\infty]]_{A}$, $W_{1}A^{-1} \leq 1$.
Then we obtain an inequality
\[ 1 \leq W_{4}^{-1} = W_{1}W_{6} \leq W_{1}(A^{-1}C) \leq W_{1}A^{-1} \leq 1 \]
Hence the all inequalities appeared in this argument must be equality.
This happens only if $A=B=C=1$, which is a contradiction.\\

\noindent
{\bf Case 2: $A\leq C \leq B$}\\

As in the Case 1, we have $W_{4} \leq  1$, $W_{5} \geq 1$ and $W_{6} \geq 1$. Since $W_{5} \in [[-1,M_{5}]]_{C^{-1}B} $ $(M_{5}<0)$, $W_{5}(C^{-1}B)^{-1} \leq 1$.
Then
\[ W_{2} = W_{5}W_{4} \leq W_{5}(B^{-1}A) = [W_{5}(C^{-1}B)^{-1}](C^{-1}A) \leq W_{5}(C^{-1}B)^{-1} \leq 1\]
hence $A \leq C \leq B \leq 1$. Then as in the Case 1, we get an inequality
\[ 1 \leq W_{4}^{-1} = W_{1}W_{6} \leq W_{1}(A^{-1}C) \leq W_{1}A^{-1} \leq 1. \]
which leads a contradiciton.\\

\noindent
{\bf Case 3: $C \leq A \leq B$}\\

In this case $W_{4} \leq 1$, $W_{5} \geq 1$ and $W_{6} \leq 1$. Then $W_{3}=W_{5}W_{6}^{-1} \geq 1$, so $1 \leq C \leq A \leq B$. Thus, $W_{1} \geq 1$.
Since $W_{3} \in [[-1, M_{3}]]_{C} $ $(-1\leq M_{3}<0)$, $W_{3} \leq C$.
However we have an inequality,
\[ B \leq W_{2} \leq W_{2}W_{1} = W_{3} \leq C \]
Thus, the all inequalities appeared in this argument must be equality.
This happens only if $A=B=C=1$, which is a contradiction.\\

Next we consider the case the assumption (2) holds. 
With no loss of generality, we may assume $A \leq B$.\\

\noindent
{\bf Case 1: $A \leq B \leq C$}\\

In this case $W_{4} \leq 1$, $W_{5} \geq 1$, and $W_{6} \geq 1$.
First of all, we determine the parity of $A, B$ and $C$.
Assume that $C < 1$. Then $W_{1} < 1$, $W_{2}<1$, $W_{3} >1$. This contradicts the grobal cycle relation   
$W_{3}=W_{2}W_{1}$. Thus $C \geq 1$. By a similar argument, we conclude $A \leq 1$. Then we get an inequality
\[ W_{2}^{-1} = W_{1}W_{3}^{-1} \geq W_{1} C = AW_{1}A^{-1}C \geq AW_{1}W_{6} = AW_{4}^{-1} \geq A A^{-1}B = B \]
hence $W_{2}^{-1} \geq B$. Since $W_{2} \in [[+1,+\infty]]_{B}$, this implies $B \leq 1$.

Recall that we have assumed $A_{6} = Q(r)$ $(-1\leq r < 0)$. Let $r= \frac{-p}{q}$ $( q, p \in \Z,\;\; q > p> 0$. Then $(W_{6}^{-1}A^{-1}C)^{p} = (A^{-1}C)^{q-p}$.
Since $(C^{-1}B)^{-1} A^{-1}C (C^{-1}B) =  B^{-1}C A^{-1}B \geq B^{-1}C \geq 1$, by Lemma \ref{lem:conjugate}, $(C^{-1}B)^{-1} (W_{6}^{-1}A^{-1}C)(C^{-1}B) \geq 1$.
Hence
\[ 1 \leq (C^{-1}B)^{-1} W_{6} C^{-1}B \leq (C^{-1}B)^{-1} A^{-1}C( C^{-1}B) = B^{-1}C A^{-1}B \]
so we get an inequality $1 \leq W_{6} C^{-1}B \leq A^{-1}B$.
Now we are ready to deduce a contradiction.
By observed inequalities, we get
\[ W_{1}^{-1} \leq W_{1}^{-1}W_{2}^{-1} = W_{3}^{-1}= W_{6}W_{5}^{-1} \leq W_{6}C^{-1}B \leq A^{-1}B \leq A^{-1} \]
This implies all inequalities appeared in this argument must be equalities, so $A=B=C=1$, which is a contradiction. \\

\noindent
{\bf Case 2:} $A \leq C \leq B$ or $C \leq A \leq B$\\

If $B < 1$, then $A,C,B <1$ so we get $W_{2} < 1$, $W_{1}<1$, and $W_{3}>1$. This contradicts the global cycle relation $W_{2}W_{1}= W_{3}$ we have $B \geq 1$.
On the other hand, in this case, $W_{4} \leq 1$, $W_{5} \leq 1$ so $W_{2}=W_{5}W_{4} \leq 1$. 
But $B\geq 1$ implies $W_{2} \geq 1$, so above inequalities must be equalities. 
This implies $A=B=C=1$, which is a contradition.

\end{proof}

Links in Theorem \ref{thm:main} contains an interesting family of knots.
Recall that an oriented knot $K$ is called {\it positive} if $K$ is represented by a diagram $D$ having only positive crossings. 

\begin{cor}
\label{cor:gtwo}
Let $K$ be a knot in $S^{3}$ which is positive and genus two.
Then the fundamental group of the double branched cover $\Sigma_{2}(K)$ is not left-orderable. 
\end{cor}
\begin{proof} 
As we have seen in Corollary \ref{cor:alter}, for an alternating link $L$, $\pi_{1}(\Sigma_{2}(L))$ is not left-orderable, so we restrict our attention to non-alternating links.
Jong-Kishimoto showed that a non-alternating positive knot of genus is represented by a diagram
obtained from three diagrams $9_{39}^{+},9_{41}^{+}$ and $12^{+}_{1202}$ by performing the $\ttwo$-moves \cite{jk}. See Figure \ref{fig:gent2}.
A knot diagram obtained from the diagram $9_{39}^{+}$ belongs to the diagrams in Theorem \ref{thm:main}  (2), and a knot diagram obtained from the diagram $9_{41}^{+}$ or $12_{1202}^{+}$ belongs to the diagrams in Theorem \ref{thm:main} (1). Thus the double branched coverings of these knots have the non-left-orderable fundamental group.
\end{proof}

\begin{figure}[htbp]
 \begin{center}
\includegraphics[width=110mm]{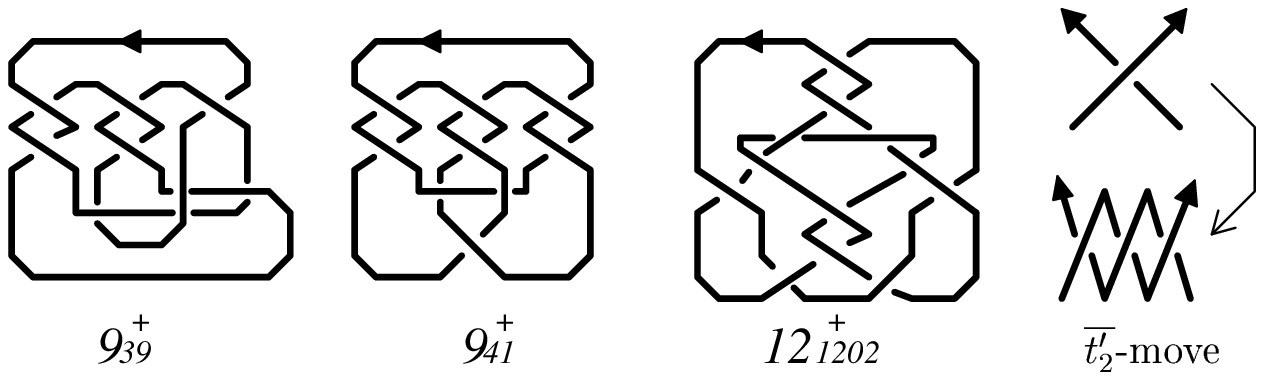}
 \end{center}
 \caption{Generator of genus two positive non-alternating knots and $\ttwo$-move}
 \label{fig:gent2}
\end{figure}

Next we give another example of links having more complicated diagram.
 \begin{thm}
 \label{thm:main2}
 Let $L$ be a link in $S^{3}$ which is represented by a diagram in Figure \ref{fig:main2}, where $A_{i}$ are algebraic tangles. Let $[[m_{i},M_{i}]]$ be the universal range of $A_{i}$. Assume the following conditions.
\begin{enumerate}
\item $A_{1}= Q(r)$ $(\Q \in r, \; r \geq 1)$ and $A_{4}=A_{10}= Q(-1)$.
\item $m_{2}, m_{3} \geq -1$.
\item $M_{2},M_{3},M_{5},M_{6},\ldots,M_{9} <0$
\end{enumerate}

Then $\pi_{1}(\Sigma_{2}(L))$ is not left-orderable.

\begin{figure}[htbp]
 \begin{center}
\includegraphics[width=110mm]{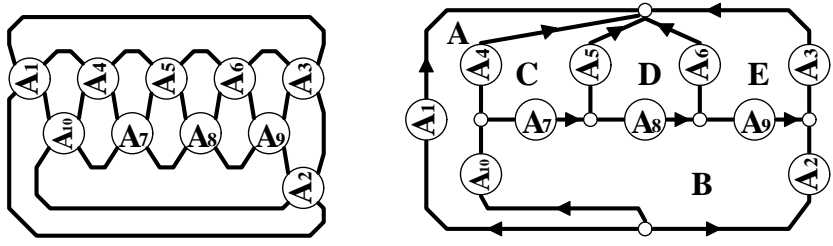}
 \end{center}
 \caption{Link diagram and the coarse decomposition graph}
 \label{fig:main2}
\end{figure}

 \end{thm}
 
 \begin{proof}
 The coarse decomposition graph of the link diagram is given in Figure \ref{fig:main2} right.
 Let us take region generators $A,B,C,D$ and $E$ as in Figure \ref{fig:main2} right and let $W_{i}=W_{A_{i}}$.
 First we write down the coarse Brunner's presentation.\\
 
\noindent 
 {\bf [Generators:]}
\[  W_{1}, W_{2},\ldots, W_{10},A,B,C,D,E\]
{\bf [Local Coarse relations:]}
\[ 
\begin{array}{ll}
W_{1} \in [[m_{1},M_{1}]]_{A} & W_{2} \in [[m_{2},M_{2}]]_{B^{-1}} \\
W_{3} \in [[m_{3},M_{3}]]_{E^{-1}} & W_{4} \in [[m_{4},M_{4}]]_{A^{-1}C} \\
W_{5} \in [[m_{5},M_{5}]]_{C^{-1}D} & W_{6} \in [[m_{6},M_{6}]]_{D^{-1}E} \\
W_{7} \in [[m_{7},M_{7}]]_{C^{-1}B} & W_{8} \in [[m_{8},M_{8}]]_{D^{-1}B}\\
W_{9} \in [[m_{9},M_{9}]]_{E^{-1}B} & W_{10} \in [[m_{10},M_{10}]]_{A^{-1}B}
\end{array}
\]
{\bf [Global cycle relations:]}
\[ 
\begin{array}{lll}
W_{1} = W_{4}W_{10}, & W_{9}W_{8}W_{7}W_{10}=W_{2}, & \\
W_{4}=W_{5}W_{7}, & W_{5}=W_{6}W_{8}, & W_{6}=W_{3}W_{9}
\end{array}
\]

Here we remove the trivial region generator to make the presentration simple.
We remark that by global cycle relations, $W_{1}=W_{3}W_{2}$ holds.
By assumption (1), we have $W_{1} = A^{m}$, $W_{4}=C^{-1}A$, and $W_{10} = B^{-1}A$.

Assume that $\pi_{1}(\Sigma_{2}(L))$ has a left ordering $<$.
With no loss of generality, we may assume $A \leq B$ holds.\\

\noindent
{\bf Case 1:} $B$ is the $<$-maximal element among the non-trivial region generators $\{A,B,C,D,E\}$.\\

In this case $W_{7},W_{8},W_{9},W_{10} \leq 1$. Since $W_{9}W_{8}W_{7}W_{10}=W_{2} $, $W_{2} \leq 1$. $W_{2} \in [[\infty,M_{2}]]_{B^{-1}}$ ($M_{2}<0$), we conclude $B \leq 1$. Thus, all region generators $A,B,C,D,E$ are either trivial or $<$-negative. 
Now $A,E \leq 1$ implies that  $W_{1} \leq 1$ and $W_{3} \leq 1$.
By the global cycle relations $W_{4}=W_{5}W_{7}$, $W_{5}=W_{6}W_{8}$ and $W_{6}=W_{3}W_{9}$, we conclude $W_{4},W_{5},W_{6} \leq 1$.
From inequalities $W_{4},W_{5},W_{6} \leq 1$, we get inequalities of region generators $A \leq C \leq D\ \leq E \leq B \leq 1$.

Now we are ready to deduce a contradiction. By hypethesis, $AW_{1}^{-1} \geq 1$, hence 
\[ 1 = W_{4}W_{10}W_{1}^{-1} = (C^{-1}A)(B^{-1}A)(W_{1}^{-1}) \geq  C^{-1}A B^{-1} \]
Thus, we get $C \geq AB^{-1}$.

In particular, $1 \geq BC \geq BAB^{-1}$, so $1 \leq BA^{-1}B^{-1}$.
Recall that we have assumed that $A_{1} = Q(p\slash q)$, $(p>q>0)$.
Thus $(A^{-1}W_{1})^{q}=A^{p-q}$.
Therefore by Lemma \ref{lem:conjugate}, $B(A^{-1}W_{1})B^{-1} \leq 1$ hence 
\[ W_{1}B^{-1} \leq AB^{-1} \leq C. \]

Finally, we observe the inequaltiy
\[ E \leq W_{3} = W_{1}W_{2}^{-1} \leq  W_{1} B^{-1} \leq C \]
Hence we conclude $C \geq E$, so the all inequality appeared in this argument must be equality.
This implies $A=B=C=D=E=1$, which is a contradiction.\\

\noindent
{\bf Case 2:} $B$ is not the $<$-maximal element among the non-trivial region generators $\{A,B,C,D,E\}$.\\

Since we have assumed $A \leq B$, we may assume that $A$ is not the $<$-maximal.
Assume that $C$ is the $<$-maximal. Then $W_{5},W_{7} \leq 1$ and $W_{4} \leq 1$.  By the global cycle relation $W_{4}=W_{5}W_{7}$, we conclude these three inequalities must be equalities.
This implies  $A=B=C=D=E=1$, which is a contradiction.

Similarly, if $D$ (resp. $E$) is $<$-maximal, then $W_{6},W_{8}>1$ and $W_{5}<1$ (resp.  $W_{3},W_{9}>1$ and $W_{6}<1$), which leads the contradiction via the global cycle relation $W_{5}=W_{6}W_{8}$ (resp. $W_{6}=W_{3}W_{9}$).
 \end{proof}

\section{Remarks on the L-space conjecture}
 
We close the paper by giving short remark on the relationships between our works and the L-space conjecture. 
  
In 3-manifold topology, it is an interesting problem to study the relationships between orderability of the fundamental groups and the topology or geometry of 3-manifolds. Boyer-Rolfsen-Wiest showed that if the fundamental group of a 3-manifold $M$ is not left-orderable, then $M$ is a rational homology 3-sphere \cite{brw}. A 3-manifold $M$ is an {\em L-space} if $M$ is a rational homology sphere and the rank of the (hat version of) Heegaard Floer homology $\widehat{HF}(M)$ is equal to $|H_{1}(M;\Z)|$, the cardinal of the 1st homology group of $M$. L-spaces include 3-manifolds having the spherical geometry, in particular, lens spaces \cite{os}. 
Recall that we have adapted the convention that the trivial group is not left-orderable, so $S^{3}$ is considered as an L-space with non-left-orderable fundamental group. 

As for the orderability of the fundamental groups of 3-manifolds, there is a remarkable conjecture:
 
\begin{conj}[L-space conjecture \cite{bgw}]
The fundamental group of a rational homology 3-sphere $M$ is non-left-orderable if and only if $M$ is an L-space.
\end{conj}

This conjecture is verified in many cases, such as Seifert Fibered spaces \cite{wa}, or more generally non-hyperbolic geometric 3-manifolds \cite{bgw}, and some other cases. There are many examples 3-manifolds having non-left-orderable fundamental groups \cite{dpt},\cite{rs},\cite{rss} and many of them are confirmed to be L-spaces \cite{cw1},\cite{cw2},\cite{p}. Conversely, many known L-spaces, such as the double branched covering of alternating links, or one obtained as some Dehn surgeries, are shown to have non-left-orderable fundamental group \cite{bgw}, \cite{gr}. (See Corollary \ref{cor:alter}).

Now let us turn to the relationships between our results and the L-space conjecture.
For a 3-manifold $M$ obtained as a double branched cover of a link, there is a useful criterion to show $M$ is an L-space: The double branched covering of a quasi-alternating link $L$ is an $L$-space \cite{os}. 

Recall that a link is called {\em quasi-alternating} if it belongs to the set $\mQ$ which is the smallest set of links characterized by the following two properties:
\begin{enumerate}
\item unknot $ \in \mQ$.
\item If $L$ has a diagram $D$ with a crossing $c$ such that
\begin{enumerate}
\item Two smoothing $D_{0}$ and $D_{\infty}$ at $c$ (See Figure \ref{fig:smooth}) represents links $L_{0}$, $L_{\infty}$ both of which belong to $\mQ$.
\item $\det(L_{0})+ \det(L_{\infty}) = \det(L)$.
\end{enumerate}
then $L$ belong to $\mQ$. Such a crossing $c$ is called a {\em quasi-alternating crossing}.

\end{enumerate} \begin{figure}[htbp]
 \begin{center}
\includegraphics[width=40mm]{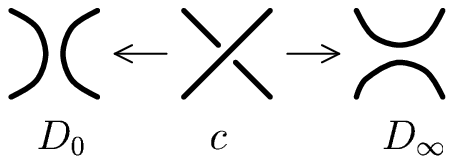}
 \end{center}
 \caption{Smoothing $D_{0}$ and $D_{\infty}$}
 \label{fig:smooth}
\end{figure}

Alternating links are quasi-alternating. In fact, all crossing points of an alternating knot diagram are quasi-alternating crossings.

Quasi-alternating links have a nice property with respect to a tangle replacement operation.
A link obtained by replacing a quasi-alternating crossing with a rational tangle which extends the crossing is also quasi-alternating \cite{ck}. By using this property of quasi-alternating links, we can confirm that many links in Theorem \ref{thm:simalter}, Theorem \ref{thm:main} and Theorem \ref{thm:main2} are quasi-alternating.
So our results provide a lot of new examples of L-spaces with non-left-orderable fundamental groups. 
For example, the positive knots of genus two are quasi-alternating \cite{jk}, hence by Corollary \ref{cor:gtwo} their double branched coverings are L-spaces having non-left orderable fundamental groups.

On the other hand, we do not know whether {\em all} links in Theorem \ref{thm:simalter}, Theorem \ref{thm:main} and Theorem \ref{thm:main2} are quasi-alternating or not. In particular we do not know whehter their double branched covering $\Sigma_{2}(L)$ are L-spaces or not. Thus our family of links also  provides a lot of candidates of the counter examples of the L-space conjecture.

\begin{rem}
\label{rem:diagram}
The diagram in Theorem \ref{thm:main} and Theorem \ref{thm:main2} are obtained by modifying quasi-alternating, but non-alternating knot diagrams.
The diagram in Theorem \ref{thm:main} is found by generalizing the diagram $9_{41}^{+}$ and $9_{39}^{+}$ in Figure \ref{fig:gent2}, and the diagram in Theorem \ref{thm:main2} is found by generalizing the quasi-alternating diagram given in \cite[Fig.4]{os2}. 

In general, by using the coarse Brunner's presentation argument, one can find other examples of links whose double branched cover have non-left-orderable fundamental group by modifying quasi-alternating links in an appropriate way (that is, by replacing each crossing with an algebraic tangle having certain universal range). The main point is that the obtained family of links might contain non-quasi-alternating links, so it is unknown whether their double branched coverings are L-spaces or not.
\end{rem}

Finally we remark that the coarse presentation method also can be applied to show non-left-prderability of 3-manifolds obtained as not only the double branched coverings, but also as Dehn surgeries, since well-known Montesinos trick \cite{mo} relates some Dehn surgeries and double branched covering constructions.
It is an interesting problem to construct a nice coarse presentation of the fundamental group of 3-manifolds obtained as other constructions, such as Dehn surgery, more general branched coverings, and Heegaard splittings.

\end{document}